%% file: paper-residue.tex
\documentclass{article}
\usepackage[svgnames]{xcolor}
\usepackage[left=2.5cm, right=2cm, top=3cm, bottom=1cm]{geometry}
\usepackage[shortlabels]{enumitem}
\usepackage{amsmath,amsthm,amssymb,fancyhdr,graphicx}
\usepackage{tikz}
\usetikzlibrary{positioning,graphs,graphs.standard}
\usepackage{fleqn}
\usepackage{enumitem}
\setlist{nolistsep}
\setlist{nosep}

\newcommand{\family}[1]{\mathcal{#1}}    
\newcommand{\partialG}{\family{A}}       
\newcommand{\fullG}{\family{B}}          
\newcommand{\splicedG}{\family{S}}       

\newcommand{\size}[1]{\text{sz}(#1)}     
\newcommand{\parity}[1]{\pi(#1)}         
\newcommand{\sizep}{\size{p}}            
\newcommand{\parityp}{\parity{p}}        

\newtheorem{theorem}{Theorem}
\newtheorem{lemma}{Lemma}

\newtheorem{prop}{Proposition}

\newcounter{enumTemp}

\begin{document}
\begin{flushleft}
\pagestyle{fancy}

\begin{center}
\huge
Non-Hamiltonian 2-regular Digraphs -- Residues \\
\large
Munagala V. S. Ramanath\\
\end{center}

\begin{abstract}
  In earlier papers, we showed a decomposition of the arcs of 2-diregular digraphs
  (2-dds) and used it to prove some conditions for these graphs to be non-Hamiltonian; we
  then extended this decomposition to a larger class of digraphs and used it to construct
  infinite families of (strongly) connected non-Hamiltonian 2-dds and provided techniques
  to establish non-Hamiltonicity in special cases. In the present paper, for a subclass
  of these graphs, we show connections between non-Hamiltonicity and sets of permutations
  in the full symmetric group S(n) by introducing the concepts of biconjugates, excluded
  sets and residues; we then use these concepts to prove a necessary and sufficient
  condition for non-Hamiltonicity.
\end{abstract}

\section{Introduction}
In [\ref{ref-ram}], we showed that the Hamiltonian Circuit (HC) problem is NP-complete
even for the family of 2-diregular digraphs ({\sl 2-dds}) using an earlier result from
[\ref{ref-plesnik}]. We then presented an efficiently computable decomposition of the
arc-set of a 2-dd into alternating cycles (ACs) and used it to prove some conditions
for these graphs to be non-Hamiltonian. For example, an infinite subfamily has the
property that all factors of graphs in this family have an even number of cycles and
therefore are non-Hamiltonian; the AC decomposition allows us to easily identify graphs
of this family. It also yields a simple algorithm to enumerate all $2^K$ factors ($K$ is
the number of ACs) of a a 2-dd. This algorithm is exponential only in $K$, so for a
fixed $K$ it yields a polynomial-time algorithm to solve HC for the family of 2-dds whose
AC count does not exceed $K$.

Given this significance of the AC-count, any technique or criterion that allows us to
simplify or reduce a 2-dd to another with a smaller $K$ while preserving Hamiltonicity
is useful. In [\ref{ref-ramroute}] we presented a few such results; these results are
also useful to construct larger non-Hamiltonian 2-dds from smaller ones. The present
paper continues the exploration of these graphs by showing a relationship of
non-Hamiltonicity with sets of permutations in the full symmetric group $S_n$ and using
it to establish a necessary and sufficient condition for non-Hamilitonicity. We then
use this characterization to present additional techniques for establishing
non-Hamiltonicity in special cases and present numerical results of exhaustive
generation of these graphs to demonstrate the effectiveness of these methods.

HC has been studied for many decades; a good summary of many classical results appear
in Chapter 10 (p. 186) of [\ref{ref-berge}]. Many of these results prove that a graph
must have an HC provided it has a large number of arcs (e.g the result of
{\sl A. Ghoula-Houri} discussed in [\ref{ref-berge}, p. 195]. As Berge observes there:
{\em ``Clearly, the greater the demi-degrees} $d_G^+(x)$ {\em and} $d_G^-(x)$ {\em of a
  1-graph are, the greater are the chances that $G$ has a hamiltonian circuit''}. Other
results conclude that if an HC exists, then there must be more than one if various
conditions are satisfied.

However, there are very few results that provide conditions for a graph $G$ to {\bf not}
have an HC, especially for regular graphs with low degree. Occasionally, individual
graphs appear as counter-examples for conjectures; for example, the Tutte graph that
disproved a long-standing conjecture of Tait; many other such counter-examples are also
discussed in [\ref{ref-berge}]. Another is the Meredith graph [\ref{ref-meredith}] which
disproved a conjecture of Crispin Nash-Williams. Our work, on the other hand, explores
non-Hamiltonicity in graphs with few arcs (a 2-dd has $2n$ arcs where $n$ is the vertex
count).

\section{Definitions}
\label{sec-def}
Let $\overline{A}_n = S_n - A_n$ denote the complement of the alternating group (i.e. the
set of odd permutations) and $C_n$ the set of  {\em cyclic permutations}
(i.e. those with a single cycle and no fixed points). For $p \in S_n$, define
$\sizep$, the {\em size} of $p$, to be the cycle count of $p$ --- so a cyclic
permutation is one whose size is 1. Define a non-empty set $P \subset S_n$ to be
{\em odd} or {\em even} according as $P \subset \overline{A}_n$ or $P \subset A_n$ and
{\sl uniform} if it is either odd or even.

For a permutation $p \in S_n$, we define $\parityp$ to be $0$ or $1$ according as $p$ is
even or odd and extend it to uniform sets: $\parity{P} = 0$ if $P$ is even, $1$
otherwise; similarly, for an integer $n$, $\parity{n} = 0$ if $n$ is even, $1$ otherwise.
It is well known that $\parityp = 0 \iff \parity{n} = \parity{\sizep}$, so for a uniform
set $P$, $\parity{\sizep}$ is the same for all $p \in P$. (Some authors use the
{\sl sign} function sgn$()$ for permutation parity but this is not suitable for us since
we want to the use the same notation for integers but the sign of an integer already has
a different meaning).

In addition to the above definitions, we use most of the definitions from
[\ref{ref-ramroute}]; they are reproduced here for convenience.

Let $G = (V,A)$ be a digraph. As usual, we say that a subgraph $H$ is a {\em component}
if it is a connected component of the underlying (undirected) graph. We use $c(G)$ to
denote the number of components of $G$.
For an arc $e = (u, v)$ of $G$, $u$ the {\it start-vertex} and $v$ the {\it end-vertex}
of $e$; $u$ is a {\em predecessor} of $v$, $v$ is a {\em successor} of $u$, $e$ is an
\emph{in-arc} of $v$ and \emph{out-arc} of $u$.

For a subset $U \subset V$, we use $U_{in}$ and $U_{out}$ for the set of arcs entering
[exiting] some vertex of $U$ from [to] $V - U$. When $U = \{v\}$, we write $v_{in}$ and
$v_{out}$.

$G$ is called a {\it k-digraph} iff for every vertex $v$ we have:
\begin{align*}
  |v_{in}| &= k \; {\rm and} \; |v_{out}| = 0; \; {\rm or} \\
  |v_{in}| &= 0 \; {\rm and} \; |v_{out}| = k; \; {\rm or} \\
  |v_{in}| &= |v_{out}| = k
\end{align*}

Suppose $F$ is a 1-digraph. Clearly, it must be a (disjoint) collection of components
each of which is either a simple cycle or a simple non-cyclic path; let $F_c$ and $F_p$
denote, respectively, the sets of these components. The {\em index} of $F$ denoted by
$i(F)$ is the number of cycles in it, i.e. $|F_c|$. We call $F$ {\it open} if
$i(F) = 0$, {\it closed} otherwise. If $F_p = \emptyset$, $F$ can be viewed as a
permutation and, in this case, $F$ is {\em even} or {\em odd} according as the
permutation is even or odd. As noted in Lemma 1 of [\ref{ref-ram}] $F$ is even iff
$|V|$ and $i(F)$ have the same parity.

Suppose $G$ is a k-digraph. A 1-digraph that is a spanning subgraph of $G$ will be called
a {\em factor} of $G$ ({\em difactor} is more appropriate but since all graphs in
this document are directed, we use the shorter term for brevity).

$V$ can be partitioned into three disjoint subsets, based on the in- and out-degree at
each vertex, called respectively the set of {\it entry}, {\it exit} and {\it saturated}
vertices:
\begin{align*}
  V_{entry} &= \{ v \in V \; | \; |v_{in}|  = 0, |v_{out}| = k \}\\
  V_{exit}  &= \{ v \in V \; | \; |v_{in}|  = k, |v_{out}| = 0 \}\\
  V_{sat}   &= \{ v \in V \; | \; |v_{in}| = |v_{out}| = k \}
\end{align*}

The entry and exit vertices are {\it unsaturated} vertices.
$G$ is {\em closed} if {\em all} of its factors are {\em closed}, {\em open} otherwise.
A Hamiltonian circuit of $G$ is clearly a (closed) factor of $G$.

A 1-digraph $F$ where $V_{entry} \neq \emptyset$ defines a unique bijection $r_F$ from
$V_{entry}$ to $V_{exit}$ where $r_F(u)$ is the unique exit vertex of $F$ where the simple
path beginning at $u$ ends. This bijection will be called the {\em route} defined by $F$.
Notice that the route is defined based solely on the extremities of the paths of $F_p$
and ignores the internal vertices of those paths as well as all of $F_c$. So distinct
factors can yield the same route. A route is {\em open} if it is defined by at least one
open factor, {\em closed} otherwise.

A k-digraph is called a {\em k-diregular digraph} ($k$-dd) if all of its vertices are
saturated; it is sometimes convenient in this case to say that $G$ is {\em saturated}.
Clearly, a component of a $k$-digraph must also be a $k$-digraph (and {\em may} be a
2-dd); likewise, a component of a $k$-dd must also be a $k$-dd.

We use $\partialG$ and $\fullG$ to denote the families of 2-digraphs
and 2-dds respectively. Clearly we have $\fullG \subset \partialG$.

We now generalize the notion of alternating cycles, defined for $\fullG$
in [\ref{ref-ram}], to $\partialG$.

Let $G = (V,A) \in \partialG$. A sequence of $2r$ {\em distinct} arcs
$X = (e_0, e_1, ..., e_{2r-1})$ in $A$ is an {\it alternating cycle (AC)}
if $e_i$ and $e_{i \oplus 1}$ have a common end-vertex [start-vertex] if $i$ is even [odd],
where $\oplus$ denotes addition mod $2r$. The arcs of an alternating cycle can be
partitioned into two disjoint sets called the set of {\em forward} and {\em backward}
arcs:
$X_f = \{ e_i \; | \; i \; \mbox{is even}\}$ and
$X_b = \{ e_i \; | \; i \; \mbox{is odd}\}$.
In [\ref{ref-ram}] we use the terms {\em clockwise} and {\em anti-clockwise} for these
arcs. The naming of the sets as ``forward'' and
``backward'' is arbitrary since traversing the arcs in the opposite direction switches
the sets. $X$ is called {\it even} or {\it odd} according as $r$ is even or odd. Clearly,
we must have $|X_f| = |X_b| = r$.

Informally, an AC is formed by starting with an arc, traversing it
forward, then traversing the next arc backward (which is guaranteed to be uniquely
possible since the end-vertex has indegree 2); this process of alternating forward
and backward traversals is continued until we return to the starting arc. So, an AC
always has an even number of arcs, half of which are forward and the rest backward.

Clearly, the arcs of $G$ can be uniquely partitioned into ACs in linear time.
As shown in Proposition~2 of [\ref{ref-ramroute}], a factor of $G$ can equivalently be
defined as a subgraph that includes all of $X_f$ and none of $X_b$ or vice versa for
every AC $X$ in $G$. A cycle [path] in $G$ is called a {\em difactorial cycle}
[{\em difactorial path}] if it is part of some factor of $G$; for any AC $X$, such a
cycle [path] cannot intersect both $X_f$ and $X_b$ (it may intersect neither).

We will henceforth represent $G \in \partialG$ by the triple $(V, A, C)$ where $C$
is the set of ACs.
The {\em index} of $G$, denoted by $i(G)$ is the smallest index among its factors, i.e.
\[
i(G) = \text{min}\{ i(F) \; | \; F \; \text{is a factor of} \; G\}
\]
Clearly, $G$ is open iff $i(G) = 0$; if $G \in \fullG$, it is Hamiltonian iff $i(G) = 1$.

Let $K \subseteq C, K \neq \emptyset$ and $G^K = (V^K, A^K, K)$ be the subgraph induced
by $K$. Clearly, $G^K \in \partialG$ and so $V^K$ is partitioned into sets
$V^K_{entry}, V^K_{exit}, V^K_{sat}$. We will use $K$ and $G^K$ interchangeably when there
is little chance of confusion. The complement $\overline{K} = C-K$ also induces a
subgraph $G^{\overline{K}}$ or simply $\overline{K}$. When $K = \{X\}$ is a singleton, we
use $K$ and $X$ interchangeably. $G^K$ is {\em proper subgraph} if $K$ is a proper
subset of $C$. {\em In the rest of this document, the term} {\bf subgraph}
{\em will almost always mean such a 2-digraph induced by a subset of $C$}.

By {\em splitting} a saturated vertex $v$ we mean replacing it with 2 vertices, $v^{in}$
and $v^{out}$, where the two arcs of $v_{in}$ [$v_{out}$] become in-arcs [out-arcs] of
$v^{in}$ [$v^{out}$]. The inverse operation of {\em splicing} consists of identifying an
entry vertex $u$ with an exit vertex $v$ to create a new saturated vertex. Both
operations yield a 2-digraph.

A subset $S \subseteq V_{sat}$ is called a {\em split-set} if splitting all the vertices
of $S$ increases the number of components. $S$ is a {\em minimal split-set} if no
proper subset of $S$ is also a split-set. When $S$ is split, the components resulting
from the split are called {\em split-components}. $G$ is called {\em k-splittable}
if it has a split-set with $k$ elements and {\em minimally k-splittable} if that split
set is minimal.

For a factor $F$ of $G$, let $\overline{F}$ denote the subgraph induced by the
complementary arc-set $A-F$. Since $G \in \partialG$, clearly, $\overline{F}$ must
also be a factor of $G$. We call it the {\em complementary factor} or simply the
{\em complement}. Clearly if $G \in \fullG$, we must have $F_p = \emptyset$ for all
factors $F$.

A 2-dd $G \in \fullG$ is {\em odd} [{\em even}] iff $i(F)$ is odd [even] for every
factor $F$ (it is possible that $G$ is neither). We call a pair of 2-dds
{\em H-equivalent} iff they are either both Hamiltonian or both non-Hamiltonian.

When discussing families of 2-digraphs, there are many combinations of properties that
are significant such as the size of ACs, whether the ACs are all odd, whether the graphs
themselves are open/closed or clean/dirty, whether all factors have odd index, etc. To
avoid verbose repetition, we use the following notation:
For any family of 2-digraphs, we use the {\bf superscript} $2k$ to denote the sub-family
where each AC has exactly $2k$ arcs and {\em clean} [{\em dirty}] the sub-family of clean
[dirty] graphs.
Similarly, we use the {\bf subscript} $m$ to denote the sub-family where each graph has
exactly $m$ ACs and finally, {\em odd} [{\em even}] to denote the sub-family where every
AC is odd [even].

A 2-dd $G \in \fullG$ is {\em odd} [{\em even}] iff $i(F)$ is odd [even] for every
factor $F$ (it is possible that $G$ is neither). We call a pair of 2-dds
{\em H-equivalent} iff they are either both Hamiltonian or both non-Hamiltonian.

So, for notational illustration, we have the family hierarchies:
\[\partialG^{2k,clean}_{m} \subset \partialG^{2k}_m \subset \partialG^{2k} \subset \partialG
\] \[
\fullG^{2k,clean}_{m,odd} \subset \fullG^{2k,clean}_{m}
\subset \fullG^{2k}_m \subset \fullG^{2k} \subset \fullG
\] \[
\fullG^{clean}_{m, odd} \subset \partialG^{clean}_{m} \enspace \mbox{and} \enspace
\fullG^{clean}_{m, even} \subset \partialG^{clean}_{m}
\]

\input{residue1.tex}

\section{Reduction}
Suppose $G = (V, A, C) \in \fullG^6$ is odd. We provided a few techniques in
[\ref{ref-ramroute}] for reducing $G$, in special cases, to either an H-equivalent 2-dd
$G' = (V', A', C') \in \fullG^6$ with $|C'| < |C|$ or to a collection of graphs
$G_i = (V_i, A_i, C_i) \in \fullG^6$ such that the collection of their AC sets
\( \{C_i\} \) is a partition of $C$ and, if $G_i$ is non-Hamiltonian for some $i$, then
$G$ must be non-Hamiltonian. During this process of reduction, the non-Hamiltonicity of
$G$ may become apparent (by one of the graphs becoming disconnected or even or containing
one of the two closed ACs $X^c_{2L}, X^c_{1L1S}$, [\ref{ref-ramroute}, Figure~1]).
Briefly these are:
\begin{enumerate}
\item
  Suppose $G$ has a dirty AC $X$; if $X$ is closed, $G$ is non-Hamiltonian; otherwise,
  the quotient $G/X$ has a unique minor $G'$; if $G'$ is disconnected, $G$ is
  non-Hamiltonian, otherwise the process can be repeated if $G'$ is also dirty
  ([\ref{ref-ramroute}, Proposition~7]).
\item
  If $G$ is 2-splittable, the unique exit/entry vertex pair of each split component can
  be spliced; if either component is even, $G$ is non-Hamiltonian; otherwise, both are
  odd and the process can be repeated if one of them is also splittable
  ([\ref{ref-ramroute}, Theorem~3]).
\item
  If $G$ has a subgraph $F$ with a unique open route, the quotient $G/F$ has a unique
  minor $G'$ that is H-equivalent to $G$; if $G'$ is disconnected, $G$ must be
  non-Hamiltonian; otherwise the process can be repeated if $G'$ also has a subgraph with
  a unique route ([\ref{ref-ramroute}, Theorem~4]).
\end{enumerate}

We now show additional techniques for reducing the AC count using
Theorem~\ref{thm-residue}.

If $F$ is a subgraph of $G$ and $\overline{F}$ its complement, let $W$ be the set of
vertices common to both and $|W| = n$. Clearly, $W$ is a split set with split components
$F$ and $\overline{F}$; by labelling these split vertices with $[1,n]$ in each split
component so that $v^{in}$ and $v^{out}$ are labelled with the same value for each
$v \in W$, we can write $G = \splicedG(F, \overline{F}, I, I)$ where $I$ is the identity
in $S_n$. Suppose $F'$ is another 2-graph with the same number of exit vertices as $F$
and $R_F = R_{F'}^{x,y}$ for some $x, y \in S_n$ where we use $R_F$ to mean the residue of
the set of open routes of $F$. By Theorem~1, Corollary (b), we can
replace $F$ with $F'$ in $G$ where each spliced exit [entry] vertex $u$ [$v$] of $F$ is
replaced by the vertex $y^{-1}(u)$ [$y^{-1}(v)$] of $F'$. The resulting 2-dd $G'$ will be
H-equivalent to $G$. This transformation is clearly useful if $F'$ has fewer ACs than
$F$ since it reduces the exponent. We have just proved:

\begin{lemma} \label{lem-replace-subgraph}
  Suppose $F$ is a subgraph of $G$ and $F'$ is an arbitrary element of $\partialG^6$
  with the same number of exit vertices as $F$ and $F' \sim F$ (i.e. $R_F$ and $R_{F'}$
  are biconjugates). We can replace $F$ with $F'$ in $G$ to get an H-equivalent graph
  $G' \in \fullG^6$.
\end{lemma}

When mapping the open routes of an open 2-graph to permutations by labeling the
exit/entry vertices with [1,n], we can, wlog, choose a labelling which makes one of the
open routes the identity. This labelling simplifies some proofs since it makes the open
route set a subset of $A_n$ containing the identity.

\begin{lemma} \label{lem-3-exit} Let $a \in A_n$ be a 3-cycle, $c \in C_n$ and $X$ a
  single clean AC with 6 arcs (i.e. $X_{clean}$ of [\ref{ref-ramroute}, Figure~1].
  \begin{enumerate}[(a)]
  \item \label{lem-A1} One of $\{ac, a^{-1}c\}$ is cyclic and the other has size 3.
  \item \label{lem-A2} The sets $E_a - C_n$ and $E_{a^{-1}} - C_n$ are the same (i.e. the
    exclusion sets of $a$ and $a^{-1}$ differ only in their subsets of cyclic
    permutations).
  \item \label{lem-A3} All three subsets
    $P_1 = \{I, a\}, P_2 = \{I, a^{-1}\}, P_3 = \{I, a, a^{-1}\}$ of $A_n$ have the same
    residue; if $n = 3$, this residue is empty.\\
  \item \label{lem-empty-residue} $R_X = \varnothing$.
  \end{enumerate}
\end{lemma}
\begin{proof}
  Suppose $a = (a_1, a_2, a_3)$ and $c = (c_1, \cdots, c_n)$; we can assume wlog that
  $c$ is rotated so that $c_1 = a_1$. There are 2 cases, depending on whether the
  orientations of $a$ and $c$ are the same or differ:
  \begin{enumerate}
  \item[] {\bf Case A}: Orientations same. So,
    \( c = (c_1 = a_1, c_2, \cdots, c_i = a_2, c_{i+1}, \cdots, c_j = a_3, c_{j+1},
            \cdots, c_n) \implies \)
    \begin{equation*}
      \begin{aligned}
        ac &= (c_1 = a_1, c_{i+1}, \cdots, c_j = a_3, c_2, \cdots, c_i = a_2, c_{j+1},
               \cdots, c_n) \; \text{and}\\
        a^{-1}c &= (c_1 = a_1, c_{j+1}, \cdots, c_n)(c_i = a_2, c_2, \cdots, c_{i-1})
                (c_j = a_3, c_{i+1}, \cdots, c_{j-1})
      \end{aligned}
    \end{equation*}
  \item[] {\bf Case B}: Orientations differ. So,
    \( c = (c_1 = a_1, c_2, \cdots, c_i = a_3, c_{i+1}, \cdots, c_j = a_2, c_{j+1},
            \cdots, c_n) \implies \)
    \begin{equation*}
      \begin{aligned}
        ac &= (c_1 = a_1, c_{j+1}, \cdots, c_n)(c_j = a_2, c_{i+1}, \cdots, c_{j-1})
              (c_i = a_3, c_2, \cdots, c_{i-1}) \; \text{and}\\
        a^{-1}c &= (c_1 = a_1, c_{i+1}, \cdots, c_{j-1}, c_j = a_2, c_2, \cdots,
                   c_i = a_3, c_{j+1}, \cdots c_n)
      \end{aligned}
    \end{equation*}
  \end{enumerate}

  So \ref{lem-A1} is proved. For \ref{lem-A2}, suppose $x \in E_a - C_n, x = a^{-1}c$
  for some $c \in C_n$. So we must have Case A above where $a^{-1}c$ has size 3 and $ac$
  is cyclic. Now, $x = a^{-1}c = a(ac) \in E_{a^{-1}} - C_n$ and \ref{lem-A2} is proved.

  \ref{lem-A3} follows from  \ref{lem-A2} by Remark~\ref{rem-residue-empty} since $I$
  excludes all of $C_n$ (if $n = 3$, $A_n$ has only 3 elements, all of which are
  excluded by all three sets).
  Finally, \ref{lem-empty-residue} follows from \ref{lem-A3} since $P_1$ is the open
  route set of $X$ (with appropriate vertex labelling).
\end{proof}

\begin{prop} \label{prop-split6}
  Suppose $G \in \partialG_{odd}$ is open and has $3$ exit and $3$ entry
  vertices and let $P$ be its open route set. Then, either $P = \{I\}$ or $P$ is one of
  the sets of Lemma~\ref{lem-3-exit}\ref{lem-A3}.
\end{prop}
\begin{proof} As observed above, we can assume the unsaturated vertices are labelled so
  that $P \subseteq A_3$ and contains the identity $I$. Since the only other elements of
  $A_3$ are the two 3-cycles that are mutual inverses, the result must hold.
\end{proof}

The proof of the next two propositions relies partly on computational verification of
some cases.
\begin{prop} \label{prop-reduce-2-AC}
  Suppose $G \in \fullG^{6,clean}$ is odd, connected and not 2-splittable and $F$ is a
  subgraph with 2 ACs and $s$ saturated vertices. If $s > 2$, we can either collapse $F$
  entirely or replace it with a single AC to get an H-equivalent odd graph
  $G' \in \fullG^{6}$.
\end{prop}
\begin{proof} Clearly, $s = 5$ is not possible, since that would make $G$ 2-splittable.
  If $s = 4$, $F$ must have a unique open route and we can replace $G$ with the unique
  minor in the quotient $G/F$ (Theorem~4 of [\ref{ref-ramroute}]). If $s = 3$, we have
  computationally verified that the set of open routes must be one of the three listed in
  Lemma~\ref{lem-3-exit}\ref{lem-A3} and hence must have empty residue; so we can
  replace it with a single clean AC since it also has empty residue by
  Lemma~\ref{lem-replace-subgraph} and Lemma~\ref{lem-3-exit}\ref{lem-empty-residue}.
\end{proof}

\begin{prop} \label{prop-reduce-3-AC}
  Suppose $G \in \fullG^{6,clean}$ is odd, connected and not 2-splittable and $F$ is a
  subgraph with 3 ACs and $s$ saturated vertices and does not have a subgraph that allows
  reduction by Proposition~\ref{prop-reduce-2-AC}.
  If $s > 4$, we can either collapse $F$
  entirely or replace it with a single AC or a pair of ACs to get an H-equivalent, odd
  graph $G' \in \fullG^{6}$.
\end{prop}
\begin{proof} 
  When $s = 7$ or $8$, the argument is similar to the previous proposition; if $s = 6$,
  by Proposition~\ref{prop-split6}, the open route set is either singleton (in which case
  $F$ can simply be collapsed) or $R_F$ is empty and can be replaced by a single clean
  AC. When $s = 5$, we have computationally verified that we must have $|R_F| = 0$ or
  $1$ (there are 85 graphs with $|R_F| = 0$ and 27 with $|R_F| = 1$). If $|R_F| = 0$, we
  can replace $F$ with the 2-AC graph $G^4$ of the previous section
  (Remark~\ref{rem-residue-empty}); if $|R_F| = 1$, we can replace it with any of the
  2-AC graphs $G^1,G^2,G^3$ of the previous section since all of them have singleton
  residue (Remark~\ref{rem-singleton-residue}).
\end{proof}

These results yield two more techniques for reducing the AC count of $G$.

To assess the effectiveness of these techniques, we present some results of generating
and analysing some of these families computationally. 

Looking at all the non-Hamiltonian connected graphs in the families $\fullG^6_k$, we find
that when $k < 6$, non-Hamiltonicity in all cases can be established by the results of
[\ref{ref-ramroute}].

When $k = 6$, as we noted in [\ref{ref-ramroute}], there are 218,161,485 non-isomorphic
graphs in $\fullG^6_6$, of which, only 12,513 are connected, clean, odd and
non-Hamiltonian. The non-Hamiltonicity of 11,320 of these can be established by
[\ref{ref-ramroute}, Theorem~3]: they are 2-splittable with both split components being
even (i.e. all factors have even size) after the requisite splicing. Of the 1,193 that
remain, Propositions~\ref{prop-reduce-2-AC} and \ref{prop-reduce-3-AC} above can
establish the non-Hamiltonicity of 1,177 leaving only 16 ``difficult'' cases (graph
$G_5$ of Figure~\ref{fig-nonham-6-ac} is one of those cases).

The family $\fullG_7^{6,clean}$ has 1,665,207 non-isomorphic connected graphs of which
209,129 are connected, clean, odd, non-2-splittable and non-Hamiltonian; of these:
\begin{enumerate}
\item 178,949 have a 2-AC subgraph with 4 saturated vertices and 23,250 have
  3 saturated vertices, so both groups can be reduced by
  Proposition~\ref{prop-reduce-2-AC} leaving 6,930.
\item 4,014 of the remainder have a 3-AC subgraph with 6 saturated vertices and can be
  reduced by Proposition~\ref{prop-reduce-3-AC} leaving 2,916.
\item {\sl All} 2,916 of these graphs have a 3-AC subgraph with 5 saturated vertices and
  so can be reduced by Proposition~\ref{prop-reduce-3-AC}.
\end{enumerate}

The family $\fullG_8^{6,clean}$ has 35,262,178 non-isomorphic graphs that are connected,
clean, odd, non-2-splittable and non-Hamiltonian; discarding those to which
Propositions~\ref{prop-reduce-2-AC} and \ref{prop-reduce-3-AC} are applicable, we are
left with 79,413 ``difficult'' graphs.

\section{References}
\begin{enumerate}
\item
  {\sl J\'an Plesn\'ik} \label{ref-plesnik}
  {\bf The NP-Completeness of the Hamiltonian Cycle Problem in
    Planar Digraphs with Degree Bound Two}, Inf. Process. Lett. 8(4): 199-201 (1979)
\item
  {\sl Claude Berge} \label{ref-berge}
  {\bf Graphs and Hypergraphs}, North-Holland Publishing Company, (1976)
\item
  {\sl G. H. J. Meredith} \label{ref-meredith}
  {\bf Regular n-valent n-connected nonHamiltonian non-n-edge-colorable graphs},
  Journal of Combinatorial Theory, Series B, Volume 14, Issue 1, February 1973, p. 55-60

\item
  {\sl Ramanath, M. V. S.} \label{ref-ram}
  {\bf Factors in a class of Regular Digraphs}, J. Graph Theory, vol. 9 (1985)
  p. 161-175.
\item
  {\sl Ramanath, M. V. S.} \label{ref-ramroute}
  {\bf Non-Hamiltonian 2-regular Digraphs} (2025)
  https://doi.org/10.48550/arXiv.2507.21381
\end{enumerate}

\clearpage
\end{flushleft}
\end{document}

%% file: residue1.tex
%
\section{Residues}\label{sec-residues}
Given a connected open 2-graph $F = (V, A, C) \in \partialG_{odd}$ with
$|V_{exit}| = n > 0$, we know from [\ref{ref-ramroute}, Corollary (a), Theorem 2] that
we can label the elements of $V_{entry}$ and $V_{exit}$ with $[1,n]$ thus mapping each
open route to a permutation in $S_n$; we also know that the set $P$ of such permutations
is uniform. Let $F' = (V', A', C')$ be another such graph and $Q$ the corresponding set
of permutations.

Suppose we splice $V_{exit}$ with $V'_{entry}$ arbitrarily (in a bijective way) and let
$x \in S_n$ represent that mapping (i.e. vertex $i$ is spliced with $x(i)$); similarly,
splice $V'_{exit}$ with $V_{entry}$ and let $y \in S_n$ represent that mapping. Clearly
this splicing yields a 2-dd in $G \in \fullG_{odd}$. Let us use $\splicedG(F, F', x, y)$
to denote this spliced 2-dd. The situation is illustrated in Figure~\ref{fig-splice}
where we've represented the sets $P$ and $Q$ of open routes in each graph by three
parallel curved arrows. We refer to $x$ and $y$ as {\em splicing permutations}.

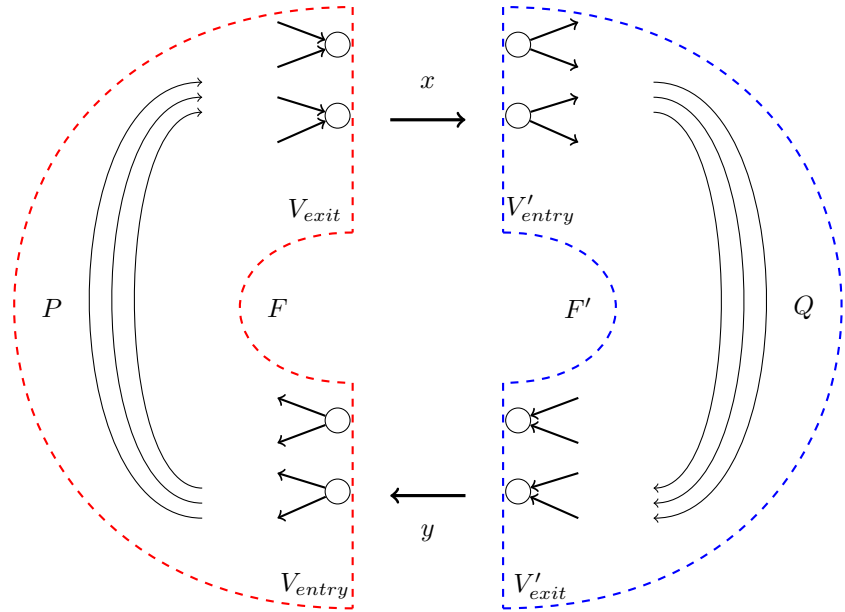
\begin{figure}[ht]
  \centering
  \input{Splice.tex}
  \caption{Splicing a pair of 2-graphs}
  \label{fig-splice}
\end{figure}

What can we say about the Hamiltonicity of $G$ ? Clearly a Hamiltonian circuit of $G$
corresponds to an element in $PxQy \cap C_n$ and conversely, any element of that
intersection corresponds to one or more Hamiltonian circuits; so $G$ is non-Hamiltonian
whenever $PxQy \cap C_n = \varnothing$. To investigate this connection
further, we introduce the notions of {\sl biconjugates}, {\sl excluded sets} and
{\sl residues} in $S_n$.

We begin with a simple proposition about permutations.
\begin{prop} \label{prop-perm-rotation}
  Suppose $a = a_1a_2...a_ia_{i+1}...a_n \in S_n$ is a product of $n$ permutations and
  $b$ is a rotated product $b = a_{i+1}...a_na_1a_2...a_i$; then, $a$ and $b$ are
  conjugates and have the same cycle type.
\end{prop}
\begin{proof}
Let $x = a_1a_2...a_i$. The result now follows since $x^{-1}ax = b$.
\end{proof}

The following remarks are now obvious (arbitrary $P,Q \subset S_n$):\\
{\sl REMARKS}:
\begin{enumerate}
  \setcounter{enumi}{\theenumTemp}
  \item $P$ and $Q$ are uniform iff $PQ$ is and $\pi(PQ)$ is $0$ or $1$ according as
    $\pi(P) = \pi(Q)$ or not.
  \item \label{rem-Cn-parity}
    $C_n$ is uniform and $\pi(C_n) \neq \pi(n)$.
  \setcounter{enumTemp}{\theenumi}
\end{enumerate}

Let $T = PxQy$. There are a some easy cases where we can show that
$T \cap C_n = \varnothing$:
\begin{prop} \label{prop-boring}
  $T \cap C_n = \varnothing$ in all of the following cases:
  \begin{enumerate}
    \item $\parity{n} = 0$, $\parity{P} = \parity{Q}$, $\parity{x} = \parity{y}$.
    \item $\parity{n} = 0$, $\parity{P} \neq \parity{Q}$, $\parity{x} \neq \parity{y}$.
    \item $\parity{n} = 1$, $\parity{P} = \parity{Q}$, $\parity{x} \neq \parity{y}$.
    \item $\parity{n} = 1$, $\parity{P} \neq \parity{Q}$, $\parity{x} = \parity{y}$.
  \end{enumerate}
\end{prop}
\begin{proof}
  The sets $C_n$ and $T$ cannot intersect in the first two cases since $\parity{C_n} = 1$
  and $\parity{T} = 0$; likewise, in the last two cases, since $\parity{C_n} = 0$ and
  $\parity{T} = 1$. Note that in all cases, $\parity{\size{t}} = 0$ for all $t \in T$.
\end{proof}

This proposition yields four sufficient conditions for the spliced graph to be
non-Hamiltonian; however these cases are not very interesting since, as observed in the
proof, the size of all elements of $T$ is even, so all factors have even size, thereby
rendering the graph trivially non-Hamiltonian and, as noted in our earlier work, this
situation is computationally easy to verify. The cases of interest are the
``opposites'' of those listed above and are enumerated in the next proposition:
\begin{prop} \label{prop-interesting}
  $\parity{C_n} = \parity{T}$ and $\size{t}$ is odd for all $t \in T \cap C_n$ in all of
  these cases:
  \begin{enumerate}
    \item $\parity{n} = 0$, $\parity{P} =    \parity{Q}, \parity{x} \neq \parity{y}$.
    \item $\parity{n} = 0$, $\parity{P} \neq \parity{Q}, \parity{x} =    \parity{y}$.
    \item $\parity{n} = 1$, $\parity{P}  =    \parity{Q}, \parity{x} =    \parity{y}$.
    \item $\parity{n} = 1$, $\parity{P}  \neq \parity{Q}, \parity{x} \neq \parity{y}$.
  \end{enumerate}
\end{prop}
\begin{proof} Obvious from discussion above.
\end{proof}

In these cases, the resulting spliced graphs may or may not be Hamiltonian and are
examined further below.

Given $Q, Q' \subseteq S_n$, we say that $Q'$ is an $(x,y)$-{\em biconjugate} (or simply
a {\em biconjugate}) of $Q$ and write $Q' = Q^{x,y}$ iff there are permutations
$x, y \in S_n$ such that $Q' = xQy$. Clearly, biconjugacy is an equivalence relation.

The following remarks follow easily:\\
{\em REMARKS}: Let $Q' = Q^{x,y}$.
\begin{enumerate}
  \setcounter{enumi}{\theenumTemp}
  \item $|Q| = |Q'|$; $Q'$ is uniform iff $Q$ is uniform.
  \item \label{rem-same-parity}
    If $Q$ is uniform, $\parity{Q} = \parity{Q'} \iff \parity{x} = \parity{y}$.
  \item \label{rem-conjugate-An} $A_n = A_n^{x,y}$ iff $\parity{x} = \parity{y}$;
    otherwise, $\overline{A}_n = A_n^{x,y}$. Similarly,
    $\overline{A}_n = \overline{A}_n^{x,y}$ iff $\parity{x} = \parity{y}$; otherwise
    $A_n = \overline{A}_n^{x,y}$.
  \item \label{rem-conjugate-Cn} $C_n$ is invariant under conjugation.
  \item Any two singleton sets $Q = \{q\}$ and $Q' = \{q'\}$ are biconjugate since
    $q' = q^{-1} q q'$. An example of a pair of subsets of $S_4$ that are even but
    {\em not} biconjugate is $P = \{I, (2 3 4)\}, P' = \{I, (12)(34)\}$.
  \item \label{rem-biconjugage-cardinality} $Q Q^{-1}$ and $Q' Q'^{-1}$ are conjugates.
  \item \label{rem-pq-qp}
    $PxQy \cap C_n = \varnothing \iff QyPx \cap C_n = \varnothing$ \\
  {\sl Proof}: Follows from the fact that for any two permutations $a, b \in S_n$,
  $ab$ and $ba$ are conjugates (Proposition \ref{prop-perm-rotation}) and hence have the
  same cycle type.

  \setcounter{enumTemp}{\theenumi}
\end{enumerate}

For any uniform $P$ and $p \in P$, define the {\em excluded set}s $E_p$ and $E_P$ as
$E_p = p^{-1} C_n$, $E_P = P^{-1} C_n$.

\begin{prop} \label{prop-exclusion-set}
  $\parity{E_P} = 1$ or $0$ according as $\parity{P} = \parity{n}$ or not.
\end{prop}
\begin{proof}
  If $\parity{P} = \parity{n}$, then either $P \subset A_n$ and $\parity{n} = 0$ or
  $P \subset \overline{A}_n$ and $\parity{n} = 1$; in either case, $\parity{E_P} = 1$
  (Remark~\ref{rem-Cn-parity}).
  If $\parity{P} \neq \parity{n}$, a similar argument shows that $\parity{E_P} = 0$.
  Conversely, if $\parity{E_P} = 0$, either $\parity{P} = 0 = \parity{C_n}$ (and
  $\parity{n} = 1$) or $\parity{P} = 1 = \parity{C_n}$ (and $\parity{n} = 0$);
  in either case, $\parity{P} \neq \parity{n}$; a similar argumeent applies if
  $\parity{E_P} = 1$.
\end{proof}

\begin{prop} \label{prop-biconjugate-exclusion}
    $E_Q^{x,y} = E_{Q^{y^{-1},x^{-1}}}$. In other words, the $(x,y)$-biconjugate of the
    exclusion set of $Q$ is the exclusion set of the $(y^{-1},x^{-1})$-biconjugate of $Q$.
\end{prop}
\begin{proof} Using Remark~\ref{rem-conjugate-Cn} we can write:
  \( E_Q^{x,y} = xQ^{-1}C_ny = (xQ^{-1}y)(y^{-1}C_ny) = (xQ^{-1}y)C_n
  =  (y^{-1}Qx^{-1})^{-1}C_n = E_{Q^{y^{-1},x^{-1}}} \).
\end{proof}

We now define the {\em residue} $R_P$ of $P$ as:
If $\parity{P} = \parity{n}$, $R_P = \overline{A}_n - E_P$; otherwise, $R_P = A_n - E_P$.

With these definitions, the following proposition is easy to see:
\begin{prop} \label{prop-residue-first}
  Suppose $P$ is uniform and $r \in S_n$.
  \begin{enumerate}[(a)]
  \item
    $r \in R_P \implies p r \notin C_n$ for all $p \in P$.
  \item
    $p r \notin C_n$ for all $p \in P \implies$ either $r \in R_P$ or $r \in A'$
    where $A' = A_n$ or $\overline{A}_n$ according as $\parity{P} = \parity{n}$ or not.
  \end{enumerate}
\end{prop}
\begin{proof}
Obvious from the definition.
\end{proof}

\begin{lemma} \label{lem-same-parity}
  Suppose $E_P = E_Q^{x,y}$. Then, $\parity{P} = \parity{Q} \iff \parity{x} = \parity{y}$.
\end{lemma}
\begin{proof}
  The conclusion follows since \( \parity{P} = \parity{Q} \iff
  \parity{P^{-1}C_n} = \parity{Q^{-1}C_n} \iff \parity{E_P} = \parity{E_Q} \iff
  \parity{x} = \parity{y} \)
\end{proof}

\begin{prop} \label{prop-biconjugate-residue}
    $R_Q^{x,y} = R_{Q^{y^{-1},x^{-1}}}$. In other words, the $(x,y)$-biconjugate of the
    residue of $Q$ is the residue of the $(y^{-1},x^{-1})$-biconjugate of $Q$.
\end{prop}
\begin{proof} Suppose $R_Q = A_n - E_Q$ and $\pi(x) \neq \pi(y)$. Then, using
  Remark~\ref{rem-conjugate-An} and Proposition~\ref{prop-biconjugate-exclusion}
  \[ LHS = x(A_n - E_Q)y = xA_ny - E_Q^{x,y} = \overline{A}_n - E_{Q^{y^{-1},x^{-1}}}
  = RHS \]
  The other cases are similar.
\end{proof}

\newcommand{\xysame}{\quad \text{if $\parity{x} = \parity{y}$,
    Lemma~\ref{lem-same-parity}}}
\newcommand{\xydiff}{\quad
  \text{otherwise, Remarks \ref{rem-conjugate-An}, \ref{rem-same-parity}}}

\begin{prop}\label{prop-residue-1}
  $R_P = R_Q^{x,y} \iff E_P = E_Q^{x,y}$
\end{prop}
\begin{proof}
  Suppose $R_P = R_Q^{x,y}$. We have two cases:
  \begin{enumerate}
  \item[] {\bf Case A}: $\parity{E_P} = 0$ (so $\parity{P} \neq \parity{n}$, by
    Proposition~\ref{prop-exclusion-set}).
    \begin{equation*}
      E_P = A_n - R_P = A_n - R_Q^{x,y} = \left\{
      \begin{aligned}
        & A_n^{x,y} - R_Q^{x,y} = (A_n - R_Q)^{x,y} = E_Q^{x,y} \xysame \\
        & \overline{A}_n^{x,y} - R_Q^{x,y} = (\overline{A}_n - R_Q)^{x,y} = E_Q^{x,y}
        \xydiff
      \end{aligned}
      \right.
    \end{equation*}
  \item[] {\bf Case B}: $\parity{E_P} = 1$ (so $\parity{P} = \parity{n}$, by
    Proposition~\ref{prop-exclusion-set}).
    \begin{equation*}
      E_P = \overline{A}_n - R_P = \overline{A}_n - R_Q^{x,y} = \left\{
      \begin{aligned}
        & \overline{A}_n^{x,y} - R_Q^{x,y} = (\overline{A}_n - R_Q)^{x,y} = E_Q^{x,y}
        \xysame \\
        & A_n^{x,y} - R_Q^{x,y} = (A_n - R_Q)^{x,y} = E_Q^{x,y} \xydiff
      \end{aligned}
      \right.
    \end{equation*}
  \end{enumerate}
  Conversely, suppose $E_P = E_Q^{x,y}$. We have two cases:
  \begin{enumerate}
  \item[] {\bf Case A}: $\parity{E_P} = 0$ (so $\parity{P} \neq \parity{n}$, by
    Proposition~\ref{prop-exclusion-set}).
    \begin{equation*}
      R_P = A_n - E_P = A_n - E_Q^{x,y}  \left\{
      \begin{aligned}
        &= A_n^{x,y} - E_Q^{x,y} = (A_n - E_Q)^{x,y} = R_Q^{x,y} \xysame \\
        &= \overline{A}_n^{x,y} - E_Q^{x,y} = (\overline{A}_n - E_Q)^{x,y} = R_Q^{x,y}
        \xydiff
      \end{aligned}
      \right.
    \end{equation*}
  \item[] {\bf Case B}: $\parity{E_P} = 1$ (so $\parity{P} = \parity{n}$, by
    Proposition~\ref{prop-exclusion-set}).
    \begin{equation*}
      R_P = \overline{A}_n - E_P = \overline{A}_n - E_Q^{x,y}  \left\{
      \begin{aligned}
        &= \overline{A}_n^{x,y} - E_Q^{x,y} = (\overline{A}_n - E_Q)^{x,y} = R_Q^{x,y}
           \xysame \\
        &= A_n^{x,y} - E_Q^{x,y} = (A_n - E_Q)^{x,y} = R_Q^{x,y} \xydiff
      \end{aligned}
      \right.
    \end{equation*}
  \end{enumerate}
\end{proof}

We now define a relation on uniform subsets $P, Q \subset S_n$ as follows:
$P$ and $Q$ are {\em equivalent} $P \sim Q$ if $R_P = R_Q^{x,y}$ for some $x,y \in S_n$.

\begin{prop} \label{prop-residue-eq-relation}
  The relation $\sim$ is an equivalence relation.
\end{prop}
\begin{proof}
  By Proposition \ref{prop-residue-1}, $P \sim Q \iff E_P = E_Q^{x,y}$, so the result
  follows since biconjugacy is an equivalence relation.
\end{proof}

The following remarks should now be obvious:

{\sl REMARKS}:
\begin{enumerate}
  \setcounter{enumi}{\theenumTemp}
  \item \label{rem-cardinality} $P \sim Q \implies |E_P| = |E_Q|$ and $|R_P| = |R_Q|$.
  \item $P \subset Q \implies E_P \subseteq E_Q$ and $R_Q \subseteq R_P$.
  \item $E_{P \cup Q} = E_P \cup E_Q$ and $R_{P \cup Q} = R_P \cap R_Q$.
  \item \label{rem-residue-size} If $P$ is singleton, $|R_P| = \frac{n!}{2} - (n-1)!$.
  \item \label{rem-residue-empty} If $P, Q$ have empty residue, $P \sim Q$.
  \item \label{rem-singleton-residue} If $R_P = \{r_1\}$ and $R_Q = \{r_2\}$ are
    singleton, $P \sim Q$, since we can choose $x = r_2, y = r_1^{-1}$.
  \setcounter{enumTemp}{\theenumi}
\end{enumerate}

\begin{prop}\label{prop-residue-2}
  $P, Q$ are biconjugates $\implies P \sim Q$
\end{prop}
\begin{proof}
  Assume $Q = P^{x,y}$. Now, by Remark~\ref{rem-conjugate-Cn}
  \begin{equation*}
    E_Q = Q^{-1} C_n = (x P y)^{-1} C_n = (y^{-1} P^{-1} x^{-1}) C_n =
    (y^{-1} P^{-1} x^{-1}) (x C_n x^{-1}) = y^{-1} P^{-1} C_n x^{-1} = y^{-1} E_P x^{-1}
  \end{equation*}
  and the result follows from Proposition \ref{prop-residue-1}.
\end{proof}

The converse of proposition \ref{prop-residue-2} may not hold: For example, in $S_4$, we
have $P \sim Q$, with $P = \{I, (1 4 3), (1 2 4)\}$ and
$Q =\{I, (2 4 3), (1 2 3), (1 4)(2 3) \}$ since
their residues are $\{(1 4)\}$ and $\{(2 3)\}$ (Remark \ref{rem-singleton-residue});
but clearly, $P$ and $Q$ are {\em not} biconjugates since their cardinalities
differ (Remark \ref{rem-biconjugage-cardinality}).

\begin{prop} \label{prop-residue}
  Suppose $T = PxQy, T' = QyPx$ and we have the conditions of
  Proposition~\ref{prop-interesting}. The following are equivalent:
  \begin{enumerate}[(a)]
  \item \label{Q-biconjugate} $Q^{x,y} \subseteq R_P$
  \item \label{P-biconjugate} $P^{y,x} \subseteq R_Q$
  \item \label{no-cyclic-perms1} $T \cap C_n = \varnothing$
  \item \label{no-cyclic-perms2} $T' \cap C_n = \varnothing$
  \end{enumerate}
\end{prop}
\begin{proof}
  The equivalence of \ref{no-cyclic-perms1} and \ref{no-cyclic-perms2} follows from
  Proposition~\ref{prop-perm-rotation}.

  For \ref{Q-biconjugate} $\implies$ \ref{no-cyclic-perms1}, suppose
  $t \in T \cap C_n$.\\
  Then, $t = pxqy, p \in P, q \in Q$, so:
  $xqy = p^{-1}t \in E_P \implies xqy \notin R_P \implies Q^{x,y} \not\subseteq R_P$.

  Similarly, for \ref{P-biconjugate} $\implies$ \ref{no-cyclic-perms2}, suppose
  $t \in T' \cap C_n$.\\
  Then, $t = qypx, p \in P, q \in Q$, so:
  $ypx = q^{-1}t \in E_Q \implies ypx \notin R_Q \implies P^{y,x} \not\subseteq R_Q$.

  For \ref{no-cyclic-perms1} $\implies$ \ref{Q-biconjugate}, suppose
  $Q^{x,y} \not\subseteq R_P$ and let
  $xqy \in Q^{x,y} - R_P$. If we can show that $xqy \in E_P$, then, for some $p \in P$ and
  $c \in C_n$ we must have $xqy = p^{-1}c$, so $c = pxqy \in C_n$ and the implication is
  established.

  We have four cases enumerated in Proposition~\ref{prop-interesting} and for each there
  are two subcases according as $\parity{P} = \parity{n}$ or not; the following table
  completes the proof in all cases using Remark~\ref{rem-Cn-parity}:
  \begin{tabular}{|p{0.3\linewidth}||p{0.3\linewidth}|p{0.3\linewidth}|}
    \hline\hline
    & $\parity{P} = \parity{n}$ & $\parity{P} \neq \parity{n}$ \\
    \hline\hline
    \( \parity{n} = 0, \parity{P} = \parity{Q}, \parity{x} \neq \parity{y},
       \parity{C_n} = 1 = \parity{T} \) &
    \( \parity{E_P} = 1, E_P = \overline{A}_n - R_P \implies xqy \in \overline{A}_n
       \implies xqy \in E_P \) &
    \( \parity{E_P} = 0, E_P = A_n - R_P \implies xqy \in A_n \implies xqy \in E_P
    \) \\
    \hline
    \( \parity{n} = 0, \parity{P} \neq \parity{Q}, \parity{x} = \parity{y},
       \parity{C_n} = 1 = \parity{T} \) &
    \( \parity{E_P} = 1, E_P = \overline{A}_n - R_P \implies xqy \in \overline{A}_n
       \implies xqy \in E_P \) &
    \( \parity{E_P} = 0, E_P = A_n - R_P \implies xqy \in A_n \implies xqy \in E_P
    \) \\
    \hline
    \( \parity{n} = 1, \parity{P} = \parity{Q}, \parity{x} = \parity{y},
       \parity{C_n} = 0 = \parity{T} \) &
    \( \parity{E_P} = 1, E_P = \overline{A}_n - R_P \implies xqy \in \overline{A}_n
       \implies xqy \in E_P \) &
    \( \parity{E_P} = 0, E_P = A_n - R_P \implies xqy \in A_n \implies xqy \in E_P
    \) \\
    \hline
    \( \parity{n} = 1, \parity{P} \neq \parity{Q}, \parity{x} \neq \parity{y},
       \parity{C_n} = 0 = \parity{T} \) &
    \( \parity{E_P} = 1, E_P = \overline{A}_n - R_P \implies xqy \in \overline{A}_n
       \implies xqy \in E_P \) &
    \( \parity{E_P} = 0, E_P = A_n - R_P \implies xqy \in A_n \implies xqy \in E_P
    \) \\
    \hline
  \end{tabular}

  We can infer \ref{no-cyclic-perms2} $\implies$ \ref{P-biconjugate} from symmetry and
  Remark \ref{rem-pq-qp}.
\end{proof}

{\bf Corollaries}:
\begin{enumerate}[(a)]
\item \label{residue-size} $|Q| > |R_P|$ (or $|P| > |R_Q|$)
  $\implies T \cap C_n \neq \varnothing$ (and $T' \cap C_n \neq \varnothing$).
\item \label{equivalent-substitution} If $Q_1 \sim Q_2$ with $R_{Q_2} = R^{a,b}_{Q_1}$ then
  $T_1 \cap C_n = \varnothing \iff T_2 \cap C_n = \varnothing$ where
  $T_1 = PxQ_1y, T_2 = PxbQ_2ay$.
\end{enumerate}
\begin{proof}
  The first follows from Remark~\ref{rem-biconjugage-cardinality}. For the second,
  observe that $R_{Q_1} = R^{a^{-1},b^{-1}}_{Q_2} = R_{Q_2^{b,a}}$ by
  Proposition~\ref{prop-biconjugate-residue}. So
  \( T_1 \cap C_n = \varnothing \iff P^{y,x} \subseteq R_{Q_1} = R_{Q_2^{b,a}} \iff
  PxQ_2^{b,a}y \cap C_n = \varnothing \iff T_2 \cap C_n = \varnothing\).
\end{proof}

Returning to the earlier discussion of splicing 2-graphs $F$ and $F'$ to get
$G = \splicedG(F, F', x, y)$, with open route sets $P, Q$ (viewed as permutations), the
importance of the residue now becomes clear when Proposition~\ref{prop-residue} is
rephrased in the context of graphs:

\begin{theorem} \label{thm-residue}
  $G$ is non-Hamiltonian iff the
  $(x,y)$-biconjugate of open routes of $F'$ is contained in $R_P$ or equivalently, the
  $(y,x)$-biconjugate of open routes of $F$ is contained in $R_Q$.
\end{theorem}

{\bf Corollaries}:
\begin{enumerate}[(a)]
\item If the number of open routes of either $F$ or $F'$ exceeds the
  cardinality of the residue of its peer, $G$ is Hamiltonian (irrespective of the
  specific splicing permutations $x$ and $y$ used).
\item If $F_1, F_2$ have equivalent residues $R_{F_2} = R_{F_1}^{a,b}$, the spliced graphs
  $G_1 = \splicedG(F, F_1, x, y)$ and $G_2 = \splicedG(F, F_2, xb, ay)$ are H-equivalent.
\end{enumerate}


As an example of the use of this theorem, Figure~\ref{fig-nonham-6-ac} shows two
connected, open 2-graphs
$G_a, G_b \in \partialG^{6,clean}_3$, each with four saturated vertices, five exit and five
entry vertices; these graphs are spliced to get a connected, clean, odd, non-Hamiltonian
2-dd $G_5 \in \fullG^6_6$. The six ACs are shown in different colors with
black ($a^n_5, 12, 10, 11, 13, a^x_2$), blue ($a^n_1, a^x_1, 12, a^x_5, a^n_2, 13$) and
teal ($a^n_3, a^x_3, 11, a^x_4, a^n_4, 10$) in $G_a$ and
brown ($b^n_5, 6, 7, b^x_2, b^n_4, b^x_4, b^n_5$),
red ($8, b^x_3, b^n_1, b^x_1, b^n_3, 9$),
and cyan ($6, b^x_5, b^n_2, 7, 9, 8$) in $G_b$.

Vertices that are saturated in both $G_a$ and $G_b$ are labelled with plain integers
6-13. The five exit vertices of $G_a$ are labelled $\{a^x_i\}_{i=1}^5$ and the five
entry vertices $\{a^n_i\}_{i=1}^5$; similarly for the exit and entry vertices of $G_b$.
Each exit vertex $a^x_i$ of $G_a$ is spliced with the corresponding entry vertex $b^n_i$
of $G_b$; likewise for the entry vertices of $G_a$ and exit vertices of $G_b$; these
spliced vertices are labelled 1-5 and $1'$-$5'$ in $G_5$. This labelling makes it easy
to see which vertices are spliced and also to read off the open routes in $G_a$ and
$G_b$ as permutations in $S_5$; it also makes the splicing permutations $x$ and $y$ of
the earlier discussion equal to the identity $I$.

The set of open routes of each ($G_a$ has six, $G_b$ has four) and their common residue
is shown in Table~\ref{tab-nonham-6-ac}. For example, in $G_a$, using the forward arcs
(solid lines) of all three ACs, we get the factor
\( \{ (a^n_1, a^x_1), (a^n_2, 13, a^x_2), (a^n_3, a^x_3), (a^n_4, 10, 11, a^x_4),
(a^n_5, 12, a^x_5)\} \) which is the identity.

Since both sets of open routes are contained in the common residue, we can conclude, by
Theorem~\ref{thm-residue}, that $G_5$ is non-Hamiltonian.

\begin{table}[htbp]
  \caption{Open routes and residues of 2-graphs in Figure~\ref{fig-nonham-6-ac}}
  \centering
  \begin{tabular}{|l|c|c|}
    \hline\hline
    & $G_a$ & $G_b$ \\
    \hline\hline
    Open routes
    & $\{I, (235), (245), (125), (13)(25), (14)(25) \}$
    & $\{I, (253), (254), (13)(25) \}$ \\
    \hline
    Common Residue & $\{I, (235), (245), (253), (254), (125), (152), $ & \\
    & $(13)(25), (14)(25), (25)(34) \}$ & \\
    \hline
  \end{tabular}
  \label{tab-nonham-6-ac}
\end{table}

\begin{figure}[htbp]
  \centering
  \input{6AC-NonH.tex}
  \caption{Connected, clean, odd, non-2-splittable, non-Hamiltonian 2-dd
    $G_5 \in \fullG^6_6$ formed by splicing 2-graphs $G_a, G_b \in \partialG^{6,clean}_3$
    with 4 saturated vertices in each.}
  \label{fig-nonham-6-ac}
\end{figure}
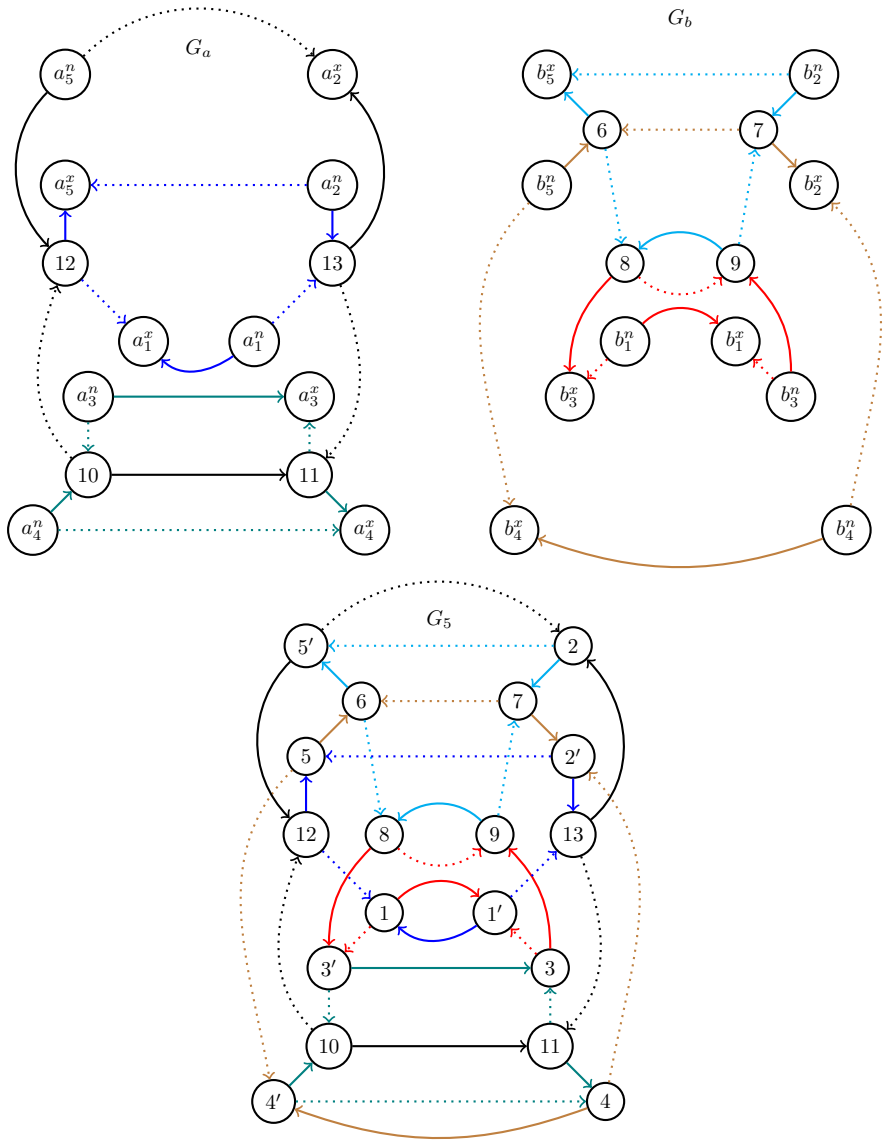

As a second example, consider the open 2-graphs $G_i$, from
$\partialG^{6,clean}_2$, each with 2 saturated vertices and 4 exit/entry vertices, shown
in Figure \ref{fig-2ac2sat}.

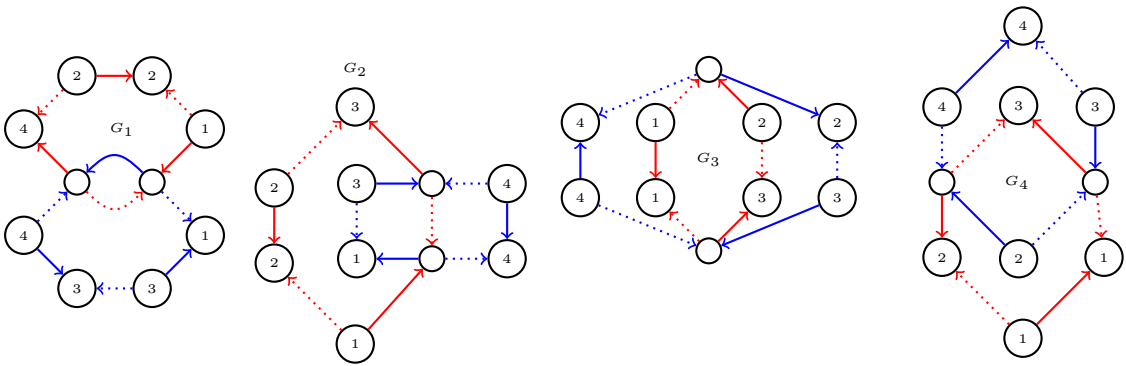
\begin{figure}[htbp]
  \centering
  \input{2AC-2Sat.tex}
  \caption{Open 2-graphs in $\partialG^{6,clean}_2$ with 2 saturated vertices}
  \label{fig-2ac2sat}
\end{figure}

\begin{table}[htbp]
  \caption{Open routes and residues of 2-graphs in Figure~\ref{fig-2ac2sat}}
  \centering
  \begin{tabular}{|l|l|l|l|l|}
    \hline\hline
    & $G_1$ & $G_2$ & $G_3$ & $G_4$ \\
    \hline\hline
    & $I$ & $I$ & $I$ & $I$ \\
    Open routes & $a = (143)$ & $a = (143)$ & $a = (14)(23)$ & $a = (12)(34)$ \\
                & $b = (124)$ & $b = (123)$ & $b = (123)$    & $b = (123)$ \\
                &             &             & $c = (243)$    & $c = (234)$ \\
    \hline
    Residue     & $\{(14)\}$  & $\{(13)\}$  & $\{(23)\}$  & $\varnothing $ \\
    \hline
  \end{tabular}
  \label{tab-2ac2sat}
\end{table}

The arcs of the two ACs are shown in red and blue with solid lines for forward arcs and
dotted for backward. Their exit and entry vertices are labelled with $1..4$ so that the
open routes can be easily mapped to elements of $S_4$; for example, in $G_1$ the open
factor obtained by choosing the solid red and dotted blue arcs becomes the identity $I$.
Their open routes and residues are shown in Table~\ref{tab-2ac2sat}. $G_1$ and $G_2$ have
3 open routes, $G_3$ and $G_4$ have 4. In all cases, the residue is either singleton or
empty.

Any pair of these graphs (including 2 copies of the same graph) can be spliced in
$4! \times 4! = 576$ ways; among such graphs, the even ones are trivially
non-Hamiltonian; the odd ones are all, by the above corollary, Hamiltonian. $G_4$ allows
a stronger claim since it has empty residue: Any {\bf open} 2-graph with 4 exit vertices
(with any number of odd ACs of any size) can be spliced with it; such a spliced graph,
if odd, must be Hamiltonian.

As a final example, there are 15 connected graphs in $\partialG^{6,clean}_3$ each with
four saturated vertices, five exit and five entry vertices, five open routes and five
elements in each residue. The open route set of one of them is
$P = \{ I, (1 2 5 4 3), (3 5 4), (1 2 4), (1 2 5) \}$ and the corresponding residue
is $R_P = \{ (2 3 5), (1 3 5), (1 2)(4 5), (1 2)(3 4), (1 2)(3 5) \}$. As expected,
since $n = 5$ is odd, $P$, $C_n$ and $R_P$, are all even and, the presence of
a cyclic permutation in $P$ excludes the identity from $R_P$ and the presence of the
identity excludes all of $C_n$. Since the cardinalities of the residues and open route
sets are equal, Corollary (a) of the above Theorem is not applicable; nevertheless, we
find that if the result of splicing any pair of these graphs (irrespective of the
splicing permutations used) is odd, it must be Hamiltonian by Theorem~\ref{thm-residue}
since the equivalence class under biconjugation of each of the open route sets is
disjoint from the set of 15 residues.

%% file: Splice.tex
\begin{tikzpicture}

  \tikzset{
    note/.style={draw=none, fill=none, align=center}
  }

  \def\xA{4} \def\yA{8}
  \draw[red,thick,dashed] (\xA,0) .. controls (-2,0) and (-2,\yA) .. (4,\yA);
  \draw[red,thick,dashed] (\xA,\yA) -- (\xA,5);
  \draw[red,thick,dashed] (\xA,5) .. controls (2,5) and (2,3) .. (\xA,3);
  \draw[red,thick,dashed] (\xA,3) -- (\xA,0);

  \def\xB{6} \def\yB{8}
  \draw[blue,thick,dashed] (\xB,0) .. controls (12,0) and (12,8) .. (\xB,8);
  \draw[blue,thick,dashed] (\xB,\yB) -- (\xB,5);
  \draw[blue,thick,dashed] (\xB,5) .. controls (8,5) and (8,3) .. (\xB,3);
  \draw[blue,thick,dashed] (\xB,3) -- (\xB,0);

  \draw[->, line width=1.2pt] (4.5,6.5) -- (5.5,6.5);    
  \draw[->, line width=1.2pt] (5.5,1.5) -- (4.5,1.5);    
  \node (x) at (5, 7) {$x$};
  \node (y) at (5, 1) {$y$};

  \node (F1) at (3, 4) {$F$}; \node (F2) at (7, 4) {$F'$};

  \node[circle, draw] (ex11) at (3.8, 7.5) {};            
  \draw[->, line width=0.8pt] (3, 7.8) -- (ex11);
  \draw[->, line width=0.8pt] (3, 7.2) -- (ex11);
  \node[circle, draw, below=0.6cm of ex11] (ex12) {};     
  \draw[->, line width=0.8pt] (3, 6.8) -- (ex12);
  \draw[->, line width=0.8pt] (3, 6.2) -- (ex12);

  \node[circle, draw] (en11) at (3.8, 2.5) {};            
  \draw[->, line width=0.8pt] (en11) -- (3, 2.8);
  \draw[->, line width=0.8pt] (en11) -- (3, 2.2);
  \node[circle, draw, below=0.6cm of en11] (en12) {};     
  \draw[->, line width=0.8pt] (en12) -- (3, 1.8);
  \draw[->, line width=0.8pt] (en12) -- (3, 1.2);

  \node (EX1) at (3.5, 5.3) {$V_{exit}$}; \node (EN1) at (3.5, 0.3) {$V_{entry}$};

  \node[circle, draw] (en21) at (6.2, 7.5) {};            
  \draw[->, line width=0.8pt] (en21) -- (7, 7.8);
  \draw[->, line width=0.8pt] (en21) -- (7, 7.2);
  \node[circle, draw, below=0.6cm of en21] (en22) {};     
  \draw[->, line width=0.8pt] (en22) -- (7, 6.8);
  \draw[->, line width=0.8pt] (en22) -- (7, 6.2);

  \node[circle, draw] (ex21) at (6.2, 2.5) {};            
  \draw[->, line width=0.8pt] (7, 2.8) -- (ex21);
  \draw[->, line width=0.8pt] (7, 2.2) -- (ex21);
  \node[circle, draw, below=0.6cm of ex21] (ex22) {};     
  \draw[->, line width=0.8pt] (7, 1.8) -- (ex22);
  \draw[->, line width=0.8pt] (7, 1.2) -- (ex22);

  \def\qxA{8} \def\qyA{7} \def\qxB{8} \def\qyB{1.2}    
  \def\cA{10} \def\cC{9.6} \def\cE{9.2}
  \draw[->] (\qxA,\qyA) .. controls (\cA,\qyA) and (\cA,\qyB) .. (\qxB,\qyB);
  \def\qxC{8} \def\qyC{6.8} \def\qxD{8} \def\qyD{1.4}    
  \draw[->] (\qxC,\qyC) .. controls (\cC,\qyC) and (\cC,\qyD) .. (\qxD,\qyD);
  \def\qxE{8} \def\qyE{6.6} \def\qxF{8} \def\qyF{1.6}    
  \draw[->] (\qxE,\qyE) .. controls (\cE,\qyE) and (\cE,\qyF) .. (\qxF,\qyF);

  \def\pxA{2} \def\pyA{7} \def\pxB{2} \def\pyB{1.2}    
  \def\cA{0} \def\cC{0.4} \def\cE{0.8}
  \draw[->] (\pxB,\pyB) .. controls (\cA,\pyB) and (\cA,\pyA) .. (\pxA,\pyA);
  \def\pxC{2} \def\pyC{6.8} \def\pxD{2} \def\pyD{1.4}    
  \draw[->] (\pxD,\pyD) .. controls (\cC,\pyD) and (\cC,\pyC) .. (\pxC,\pyC);
  \def\pxE{2} \def\pyE{6.6} \def\pxF{2} \def\pyF{1.6}    
  \draw[->] (\pxF,\pyF) .. controls (\cE,\pyF) and (\cE,\pyE) .. (\pxE,\pyE);

  \node (P) at (0, 4) {$P$}; \node (Q) at (10, 4) {$Q$};

  \node (EX2) at (6.5, 5.3) {$V'_{entry}$}; \node (EN2) at (6.5, 0.3) {$V'_{exit}$};

\end{tikzpicture}

%% file: 6AC-NonH.tex
\begin{tikzpicture}[->, thick, scale=0.8, every node/.style={scale=0.8}, main/.style = {circle}]
  \tikzset{node distance = 1.3cm and 1.3cm}



  \node[main] (O) at (0, 0) {};    

  \node[main] [draw, above left of=O] (n8) {$8$};
  \node[main] [draw, above right of=O] (n9) {$9$};
  \node[main] [draw, left of=n8] (n12) {$12$};
  \node[main] [draw, right of=n9] (n13) {$13$};
  \node[main] [draw, above of=n12] (n5) {$5$};
  \node[main] [draw, above of=n13] (n2p) {$2'$};
  \node[main] [draw, above right of=n5] (n6) {$6$};
  \node[main] [draw, above left of=n2p] (n7) {$7$};
  \node[main] [draw, above left of=n6] (n5p) {$5'$};
  \node[main] [draw, above right of=n7] (n2) {$2$};

  \node[main] [draw, below of=n8] (n1) {$1$};
  \node[main] [draw, below of=n9] (n1p) {$1'$};
  \node[main] [draw, below left of=n1] (n3p) {$3'$};
  \node[main] [draw, below right of=n1p] (n3) {$3$};
  \node[main] [draw, below of=n3p] (n10) {$10$};
  \node[main] [draw, below of=n3] (n11) {$11$};
  \node[main] [draw, below left of=n10] (n4p) {$4'$};
  \node[main] [draw, below right of=n11] (n4) {$4$};

  \draw[red, dotted] (n8) to [out=-45,in=225] (n9);
  \draw[red, dotted] (n1) to (n3p);
  \draw[red, dotted] (n3) to (n1p);
  \draw[red] (n8) to [out=-135,in=90] (n3p);
  \draw[red] (n3) to [out=90,in=-45] (n9);
  \draw[red] (n1) to [out=45,in=135] (n1p);

  \draw[teal, dotted] (n3p) to (n10);
  \draw[teal, dotted] (n11) to (n3);
  \draw[teal, dotted] (n4p) to (n4);
  \draw[teal] (n3p) to (n3);
  \draw[teal] (n4p) to (n10);
  \draw[teal] (n11) to (n4);

  \draw[brown, dotted] (n5) to [out=-135,in=100] (n4p);
  \draw[brown, dotted] (n4) to [out=80,in=-45] (n2p);
  \draw[brown, dotted] (n7) to (n6);
  \draw[brown] (n7) to (n2p);
  \draw[brown] (n5) to (n6);
  \draw[brown] (n4) to [out=-160,in=-20] (n4p);

  \draw[cyan, dotted] (n9) to (n7);
  \draw[cyan, dotted] (n6) to (n8);
  \draw[cyan, dotted] (n2) to (n5p);
  \draw[cyan] (n2) to (n7);
  \draw[cyan] (n6) to (n5p);
  \draw[cyan] (n9) to [out=135,in=45] (n8);

  \draw[blue, dotted] (n2p) to (n5);
  \draw[blue, dotted] (n12) to (n1);
  \draw[blue, dotted] (n1p) to (n13);
  \draw[blue] (n12) to (n5);
  \draw[blue] (n2p) to (n13);
  \draw[blue] (n1p) to [out=-145,in=-45] (n1);

  \draw[black, dotted] (n5p) to [out=45,in=135] (n2);
  \draw[black, dotted] (n13) to [out=-70,in=45] (n11);
  \draw[black, dotted] (n10) to [out=135,in=-110] (n12);
  \draw[black] (n5p) to [out=-135,in=135] (n12);
  \draw[black] (n13) to [out=40,in=-45]  (n2);
  \draw[black] (n10) to (n11);

  \node[main] (G5) at (0, 4.5) {$G_5$};    

  \node[main] (O1) at (-4, 9.5) {};    

  \node[main] [above left of=O1] (a8) {};          
  \node[main] [above right of=O1] (a9) {};         
  \node[main] [draw, left of=a8] (a12) {$12$};
  \node[main] [draw, right of=a9] (a13) {$13$};
  \node[main] [draw, above of=a12] (ax5) {$a^x_5$};
  \node[main] [draw, above of=a13] (an2) {$a^n_2$};
  \node[main] [above right of=ax5] (a6) {};          
  \node[main] [above left of=an2] (a7) {};           
  \node[main] [draw, above left of=a6] (an5) {$a^n_5$};
  \node[main] [draw, above right of=a7] (ax2) {$a^x_2$};

  \node[main] [draw, below of=a8] (ax1) {$a^x_1$};
  \node[main] [draw, below of=a9] (an1) {$a^n_1$};
  \node[main] [draw, below left of=ax1] (an3) {$a^n_3$};
  \node[main] [draw, below right of=an1] (ax3) {$a^x_3$};
  \node[main] [draw, below of=an3] (a10) {$10$};
  \node[main] [draw, below of=ax3] (a11) {$11$};
  \node[main] [draw, below left of=a10] (an4) {$a^n_4$};
  \node[main] [draw, below right of=a11] (ax4) {$a^x_4$};

  \draw[blue, dotted] (an2) to (ax5);
  \draw[blue, dotted] (a12) to (ax1);
  \draw[blue, dotted] (an1) to (a13);
  \draw[blue] (a12) to (ax5);
  \draw[blue] (an2) to (a13);
  \draw[blue] (an1) to [out=-145,in=-45] (ax1);

  \draw[black, dotted] (an5) to [out=45,in=135] (ax2);
  \draw[black, dotted] (a13) to [out=-70,in=45] (a11);
  \draw[black, dotted] (a10) to [out=135,in=-110] (a12);
  \draw[black] (an5) to [out=-135,in=135] (a12);
  \draw[black] (a13) to [out=40,in=-45]  (ax2);
  \draw[black] (a10) to (a11);

  \draw[teal, dotted] (an3) to (a10);
  \draw[teal, dotted] (a11) to (ax3);
  \draw[teal, dotted] (an4) to (ax4);
  \draw[teal] (an3) to (ax3);
  \draw[teal] (an4) to (a10);
  \draw[teal] (a11) to (ax4);

  \node[main] (Ga) at (-4, 14) {$G_a$};    

  \node[main] (O2) at (4, 9.5) {};    

  \node[main] [draw, above left of=O2] (b8) {$8$};
  \node[main] [draw, above right of=O2] (b9) {$9$};
  \node[main] [left of=b8] (b12) {};               
  \node[main] [right of=b9] (b13) {};              
  \node[main] [draw, above of=b12] (bn5) {$b^n_5$};
  \node[main] [draw, above of=b13] (bx2) {$b^x_2$};

  \node[main] [draw, above right of=bn5] (b6) {$6$};
  \node[main] [draw, above left of=bx2] (b7) {$7$};
  \node[main] [draw, above left of=b6] (bx5) {$b^x_5$};
  \node[main] [draw, above right of=b7] (bn2) {$b^n_2$};

  \node[main] [draw, below of=b8] (bn1) {$b^n_1$};
  \node[main] [draw, below of=b9] (bx1) {$b^x_1$};
  \node[main] [draw, below left of=bn1] (bx3) {$b^x_3$};
  \node[main] [draw, below right of=bx1] (bn3) {$b^n_3$};
  \node[main] [below of=bx3] (b10) {};               
  \node[main] [below of=bn3] (b11) {};               
  \node[main] [draw, below left of=b10] (bx4) {$b^x_4$};
  \node[main] [draw, below right of=b11] (bn4) {$b^n_4$};

  \draw[red, dotted] (b8) to [out=-45,in=225] (b9);
  \draw[red, dotted] (bn1) to (bx3);
  \draw[red, dotted] (bn3) to (bx1);
  \draw[red] (b8) to [out=-135,in=90] (bx3);
  \draw[red] (bn3) to [out=90,in=-45] (b9);
  \draw[red] (bn1) to [out=45,in=135] (bx1);

  \draw[brown, dotted] (bn5) to [out=-135,in=100] (bx4);
  \draw[brown, dotted] (bn4) to [out=80,in=-45] (bx2);
  \draw[brown, dotted] (b7) to (b6);
  \draw[brown] (b7) to (bx2);
  \draw[brown] (bn5) to (b6);
  \draw[brown] (bn4) to [out=-160,in=-20] (bx4);

  \draw[cyan, dotted] (b9) to (b7);
  \draw[cyan, dotted] (b6) to (b8);
  \draw[cyan, dotted] (bn2) to (bx5);
  \draw[cyan] (bn2) to (b7);
  \draw[cyan] (b6) to (bx5);
  \draw[cyan] (b9) to [out=135,in=45] (b8);

  \node[main] (Gb) at (4, 14.5) {$G_b$};    

\end{tikzpicture}

%% file: 2AC-2Sat.tex
\begin{tikzpicture}[->, node distance={10mm}, thick, main/.style = {circle}]    %
\tikzstyle{every text node part}=[font=\tiny, inner sep=.3]

\node[main] (a5) [draw] {}; 
\node[main] (a6) [draw, right of=a5] {}; 
\node[main] (a1) [draw, above right of=a6] {$1$};
\node[main] (a4) [draw, above left of=a5] {$4$};
\node[main] (a-n2) [draw, above right of=a4] {$2$};  
\node[main] (a-x2) [draw, right of=a-n2] {$2$};      
\node[main] (b1) [draw, below right of=a6] {$1$}; 
\node[main] (b4) [draw, below left of=a5] {$4$}; 
\node[main] (b-n3) [draw, below left of=b1] {$3$};  
\node[main] (b-x3) [draw, left of=b-n3] {$3$};      
\node (G1) at (0.6, 0.7) [font=\small] {$G_1$};

\draw[red] (a1) -> (a6);
\draw[red] (a5) -> (a4);
\draw[red] (a-n2) -> (a-x2);
\draw[red, dotted] (a-n2) -> (a4);
\draw[red, dotted] (a1) -> (a-x2);
\draw[red, dotted] (a5) to [out=315,in=-135,looseness=1.5] (a6);
\draw[blue] (a6) to [out=135,in=45,looseness=1.5] (a5);
\draw[blue] (b4) -> (b-x3);
\draw[blue] (b-n3) -> (b1);
\draw[blue, dotted] (a6) -> (b1);
\draw[blue, dotted] (b4) -> (a5);
\draw[blue, dotted] (b-n3) -> (b-x3);

\node (G2) at (3.7, 1.5) [font=\small] {$G_2$};
\node[main] at (3.7,1) (a2-x3) [draw] {$3$};
\node[main] (a2-n2) [draw, below left=1cm of a2-x3] {$2$};
\node[main] (a2-x2) [draw, below of=a2-n2] {$2$};
\node[main] (a2-n1) [draw, below right=1cm of a2-x2] {$1$};
\node[main] (a25) [draw, below right=1cm of a2-x3] {};
\node[main] (a26) [draw, below of=a25] {};
\node[main] (a2-n3) [draw, left of=a25] {$3$};
\node[main] (a2-n4) [draw, right of=a25] {$4$};
\node[main] (a2-x1) [draw, left of=a26] {$1$};
\node[main] (a2-x4) [draw, right of=a26] {$4$};
\draw[red] (a25) -> (a2-x3);
\draw[red] (a2-n2) -> (a2-x2);
\draw[red] (a2-n1) -> (a26);
\draw[red, dotted] (a25) -> (a26);
\draw[red, dotted] (a2-n1) -> (a2-x2);
\draw[red, dotted] (a2-n2) -> (a2-x3);
\draw[blue] (a2-n3) -> (a25);
\draw[blue] (a2-n4) -> (a2-x4);
\draw[blue] (a26) -> (a2-x1);
\draw[blue, dotted] (a26) -> (a2-x4);
\draw[blue, dotted] (a2-n4) -> (a25);
\draw[blue, dotted] (a2-n3) -> (a2-x1);

\node (G3) at (8.4, 0.3) [font=\small] {$G_3$};
\node[main] at (8.4, 1.5) (a35) [draw] {};
\node[main] (a3-n1) [draw, below left of=a35] {$1$};
\node[main] (a3-x1) [draw, below of=a3-n1] {$1$};
\node[main] (a3-n2) [draw, below right of=a35] {$2$};
\node[main] (a3-x3) [draw, below of=a3-n2] {$3$};
\node[main] (a36) [draw, below right of=a3-x1] {};
\node[main] (a3-x4) [draw, left of=a3-n1] {$4$};
\node[main] (a3-n4) [draw, below of=a3-x4] {$4$};
\node[main] (a3-x2) [draw, right of=a3-n2] {$2$};
\node[main] (a3-n3) [draw, below of=a3-x2] {$3$};
\draw[red] (a3-n2) -> (a35);
\draw[red] (a3-n1) -> (a3-x1);
\draw[red] (a36) -> (a3-x3);
\draw[red, dotted] (a3-n1) -> (a35);
\draw[red, dotted] (a3-n2) -> (a3-x3);
\draw[red, dotted] (a36) -> (a3-x1);
\draw[blue] (a35) -> (a3-x2);
\draw[blue] (a3-n3) -> (a36);
\draw[blue] (a3-n4) -> (a3-x4);
\draw[blue, dotted] (a35) -> (a3-x4);
\draw[blue, dotted] (a3-n4) -> (a36);
\draw[blue, dotted] (a3-n3) -> (a3-x2);

\node (G4) at (12.5, 0) [font=\small] {$G_4$};
\node[main] at (11.5, 0) (a45) [draw] {};
\node[main] (a4-n4) [draw, above of=a45] {$4$};
\node[main] (a4-x4) [draw, above right=1cm of a4-n4] {$4$};
\node[main] (a4-n2) [draw, below right=1cm of a45] {$2$};
\node[main] (a46) [draw, above right=1cm of a4-n2] {};
\node[main] (a4-n3) [draw, above of=a46] {$3$};
\node[main] (a4-x2) [draw, below of=a45] {$2$};
\node[main] (a4-n1) [draw, below right=1cm of a4-x2] {$1$};
\node[main] (a4-x1) [draw, above right=1cm of a4-n1] {$1$};
\node[main] (a4-x3) [draw, above right=1cm of a45] {$3$};
\draw[blue] (a4-n4) -> (a4-x4);
\draw[blue] (a4-n3) -> (a46);
\draw[blue] (a4-n2) -> (a45);
\draw[blue, dotted] (a4-n4) -> (a45);
\draw[blue, dotted] (a4-n2) -> (a46);
\draw[blue, dotted] (a4-n3) -> (a4-x4);
\draw[red] (a45) -> (a4-x2);
\draw[red] (a4-n1) -> (a4-x1);
\draw[red] (a46) -> (a4-x3);
\draw[red, dotted] (a45) -> (a4-x3);
\draw[red, dotted] (a46) -> (a4-x1);
\draw[red, dotted] (a4-n1) -> (a4-x2);

\end{tikzpicture}